\documentclass[a4paper,10pt]{amsart}

\usepackage{amssymb,amsmath,amsfonts,amsthm,mathtools}
\usepackage{latexsym,mathrsfs,pb-diagram}
\usepackage[pdftex]{graphicx}
\usepackage[all]{xy}
\usepackage[hmargin=2.5cm,vmargin=3.5cm]{geometry}
\usepackage[plainpages=false,colorlinks,pdfpagelabels]{hyperref}
\hypersetup{
  urlcolor=black,
  citecolor=green,
  linkcolor=blue
}

\numberwithin{equation}{section}

\theoremstyle{definition}
\newtheorem{Def}{Definition}[section]

\theoremstyle{remark}

\newtheorem{Rem}[Def]{Remark}

\theoremstyle{plain}
\newtheorem{Prop}[Def]{Proposition}
\newtheorem{Cor}[Def]{Corollary}
\newtheorem{Thm}[Def]{Theorem}
\newtheorem{Lem}[Def]{Lemma}

\newcommand{\dfn}{\mathrel{\dot{=}}}

\newcommand{\st}{ \ ; \ }
\newcommand{\rarr}{\rightarrow}
\newcommand{\sset}{\subset}

\newcommand{\eset}{\emptyset}
\newcommand{\Z}{\mathbb{Z}}
\newcommand{\N}{\mathbb{N}}
\newcommand{\R}{\mathbb{R}}

\newcommand{\C}{\mathbb{C}}

\newcommand{\TR}[5]{\begin{array}{c c c c c}
    {#1} & : & {#3} & \longrightarrow & {#5}\\
    & & {#2} & \longmapsto & {#4}
  \end{array}
}

\DeclareMathOperator{\Span}{\mathrm{span}}

\DeclareMathOperator{\tr}{\mathrm{trace}}

\newcommand{\del}{\partial}
\newcommand{\dd}{\mathrm{d}}
\renewcommand{\Re}{\mathsf{Re}}
\renewcommand{\Im}{\mathsf{Im}}

\newcommand{\e}{\mathbf{e}}

\DeclareMathOperator{\supp}{\mathrm{supp}}

\DeclareMathOperator{\GL}{\mathrm{GL}}

\DeclareMathOperator{\gl}{\mathfrak{gl}}

\newcommand{\gr}[1]{\mathfrak{#1}}
\DeclareMathOperator{\Ad}{\mathrm{Ad}}
\DeclareMathOperator{\ad}{\mathrm{ad}}

\newcommand{\cv}[1]{\mathbf{#1}}

\author{Gabriel Ara{\'u}jo}
\address{University of S{\~a}o Paulo, ICMC-USP, S{\~a}o Carlos, SP, Brazil}
\email{gccsa@icmc.usp.br}
\thanks{This work was partially supported by Conselho Nacional de Desenvolvimento Cient{\'i}fico e Tecnol{\'o}gico (CNPq, grants~131876/2010-4 and~140838/2012-0) and the S{\~a}o Paulo Research Foundation (FAPESP, grant~2018/12273-5).}


\title[Periodic trajectory tracking on Lie groups]{Periodic trajectory tracking for control-affine driftless systems on compact Lie groups}
\begin{document}

\begin{abstract} We treat the \emph{periodic trajectory tracking problem}: given a periodic trajectory of a control-affine, left-invariant driftless system in a compact and connected Lie group $G$ and an initial condition in $G$, find another trajectory of the system satisfying the initial condition given and that asymptotically tracks the periodic trajectory. We solve this problem locally (for initial conditions in a neighborhood of some point of the periodic trajectory) when $G$ is semisimple and the system is Lie-determined (i.e.~controllable), and only for a class of periodic trajectories (which we call \emph{regular}). Finally we present a set of sufficient conditions to ensure the existence of such trajectories.
\end{abstract}

\maketitle

\section{Introduction}

The present work addresses the problem of periodic trajectory tracking for control-affine driftless systems, specifically in the case when the ambient manifold is a Lie group $G$ (which we will further assume to be compact and connected) and the system is left-invariant (see below). It is heavily inspired by~\cite{original} (see also its first author's PhD thesis~\cite{silveira_thesis}), in which the problem is studied in $\mathrm{SU}(n)$ aiming applications to Quantum Computing, and can indeed be considered as a (tentative) extension of their methods to abstract Lie groups. We do not, however, rely on any of their results or even notations directly, but rather on their ideas; nor we aim at any applications whatsoever.

In Section~\ref{sec:pmpp} we describe the periodic trajectory tracking problem (PTTP) for our system~\eqref{eq:lids} and reduce it to the problem of stabilization of an auxiliary system~\eqref{eq:mod_sys}. The main conclusion here is that if the identity element of $e$ of $G$ is a critical point, and moreover a local attractor of this new system, then one can solve the PTTP locally i.e.~for initial conditions close to the reference periodic trajectory. This leads us to investigate some aspects of the stability of time-dependent vector fields on compact Riemannian manifolds, which we do in Section~\ref{sec:stability}, and then apply our conclusions to characterize the $\omega$-limit points of an auxiliary vector field $W$ associated to~\eqref{eq:mod_sys}: they are precisely the critical points of $W$. We also conclude that every central point of $G$ is critical to $W$, so a necessary condition for our approach to work is that $G$ is semisimple e.g.~$\mathrm{SU}(n)$.

In Section~\ref{sec:regtraj} we restrict our attention to a class of periodic trajectories, which we call \emph{regular}, for which an even simpler characterization of the critical points of the associated $W$ is achieved: they are the critical points of a Lyapunov-like function $V$; and moreover central points of $G$ are non-degenerate critical points of $V$ provided $G$ is semisimple. A little more effort then allows us to conclude that, in that case, the latter points are also local attractors of $W$, and since $e$ is obviously central we solve the PTTP locally. We close this work (Section~\ref{sec:existence}) discussing a condition that ensures the existence of periodic trajectories, including a more or less concrete construction of them.

We refer the reader to~\cite{sachkov07} and~\cite{jurdjevic_gct} for the basics of Control Theory on Lie groups. For more sophisticate aspects of Lie group theory -- notably some results regarding the adjoint representation of $G$, to which we are naturally led by the change of variables that produces the auxiliary system~\eqref{eq:mod_sys} and that stalks us until the end, revealing how semisimplicity is an essential feature to the problem -- the reader is referred to a less introductory text on the subject e.g.~\cite{knapp_lgbi}; more paramount results and definitions, as well as possibly non-standard notation, are also briefly explained in the footnotes.

\subsection*{Acknowledgments} I wish to thank H.~B.~Silveira and P.~A.~Tonelli for reading the original manuscript and making invaluable suggestions, and the latter also for many long discussions and for proposing the problem.

\section{The periodic trajectory tracking problem} \label{sec:pmpp}

Let $G$ be a compact, connected Lie group, whose Lie algebra of left-invariant vector fields we denote by $\gr{g}$. Given $X_1, \ldots, X_m \in \gr{g}$ we consider the left-invariant driftless system
\begin{align}
  \cv{x}' &= \sum_{k = 1}^m u_k X_k(\cv{x}) \label{eq:lids}
\end{align}
where $u_1, \ldots, u_m \in \R$ are controls. We shall work exclusively with smooth trajectories: $(m + 1)$-uples $(\cv{x}, u_1, \ldots, u_m)$ where $u_1, \ldots, u_m: \R \rarr \R$ are smooth (i.e.~$C^\infty$) functions -- the controls -- and $\cv{x}: \R \rarr G$ is an integral curve of the time-dependent vector field
\begin{align*}
  X(t,x) &\dfn \sum_{k = 1}^m u_k(t) X_k(x).
\end{align*}
The trajectory is said to be \emph{$T$-periodic} ($T > 0$) provided $\cv{x}, u_1, \ldots, u_m$ are $T$-periodic functions.

For simplicity, we shall assume that $\Gamma \dfn \Span \{ X_1, \ldots, X_m \}$ is bracket-generating i.e.~the Lie algebra generated by $\Gamma$ is $\gr{g}$, and hence $\Gamma$ has a single orbit thanks to Sussmann's Theorem.

\begin{Def} The \emph{periodic trajectory tracking problem} (PTTP) for system~\eqref{eq:lids} is stated as follows: given a $T$-periodic reference trajectory $(\cv{x}_r, u_1^r, \ldots, u_m^r)$ and an initial state $x_0 \in G$, find another (non-periodic) trajectory $(\cv{x}, u_1, \ldots, u_m)$ of~\eqref{eq:lids} such that $\cv{x}(0) = x_0$ and\footnote{$e$: the identity element of $G$.}
  \begin{align*}
    \lim_{t \to \infty} \cv{x}(t) \cdot \cv{x}_r(t)^{-1} &= e.
  \end{align*}
  We call $\cv{x} \cdot \cv{x}_r^{-1}$ the \emph{tracking error} between the two trajectories.
\end{Def}

\begin{Rem} The motivation for our definition of tracking error comes from concrete examples. If $G$ is a subgroup of $\GL(n, \C)$ and $\cv{x}, \cv{y}: \R \rarr G$ are curves then clearly
  \begin{align*}
    \lim_{t \to \infty} \cv{x}(t) \cdot \cv{y}(t)^{-1} = e &\Longleftrightarrow \lim_{t \to \infty} \| \cv{x}(t) - \cv{y}(t) \| = 0
  \end{align*}
  where $\| \cdot \|$ is any matrix norm.

  Moreover, consider the following asymptotic controllability problem (also sometimes called the \emph{$T$-sampling stabilization problem}, see for instance~\cite{spr14}) for system~\eqref{eq:lids}:
  \begin{quote}
    Given an initial state $x_0 \in G$ and a target state $x_\infty \in G$, find a trajectory $(\cv{x}, u_1, \ldots, u_m)$ of~\eqref{eq:lids} such that, for some $T > 0$, we have
    \begin{align*}
      \lim_{k \to \infty} \cv{x}(kT) &= x_\infty.
    \end{align*}
  \end{quote}
  It is clear that this problem can be solved if we are able to find
  \begin{enumerate}
  \item $(\cv{x}_r, u_1^r, \ldots, u_m^r)$ a periodic reference trajectory for~\eqref{eq:lids} with $\cv{x}_r(0) = x_\infty$ and
  \item $(\cv{x}, u_1, \ldots, u_m)$ a trajectory of~\eqref{eq:lids} that tracks $(\cv{x}_r, u_1^r, \ldots, u_m^r)$ i.e.~solving the PTTP.
  \end{enumerate}
  While the second question above is the main subject of the present paper, we shall discuss the first one -- the existence of periodic reference trajectories passing through arbitrary points of $G$ -- in Section~\ref{sec:existence}.
\end{Rem}

The very definition of the tracking error suggests that we can reduce the PTTP associated to a given reference trajectory $(\cv{x}_r, u_1^r, \ldots, u_m^r)$ to a stabilization problem, via a change of coordinates that we describe below. From now on we denote
\begin{align*}
  x_\infty &\dfn \cv{x}_r(0).
\end{align*}

\begin{Prop} \label{prop:mod_sys} Assume that $(\cv{z}, v_1, \ldots, v_m)$ is a trajectory of the system\footnote{The adjoint map $\Ad: G \rarr \GL(\gr{g})$ is the group homomorphism that associates to each $x \in G$ an invertible linear map $\Ad(x): \gr{g} \rarr \gr{g}$ as follows: if $I_x$ stands for the map $y \in G \mapsto x \cdot y \cdot x^{-1} \in G$ then $\Ad(x)$ corresponds to $\dd (I_x)_e: T_eG \rarr T_eG$ via the canonical isomorphism $\gr{g} \cong T_eG$.}
  \begin{align}
    \cv{z}' &= \sum_{k = 1}^m v_k \Ad(\cv{x}_r(t)) X_k(\cv{z}) \label{eq:mod_sys}
  \end{align}
  such that $\cv{z}(0) = x_0 \cdot x_\infty^{-1}$ and $\lim_{t \to \infty} \cv{z}(t) = e$. If we define
  \begin{align}
    \cv{x} &\dfn \cv{z} \cdot \cv{x}_r, \label{eq:x_track} \\
    u_k &\dfn v_k + u_k^r, \quad k \in \{1, \ldots, m\} \label{eq:u_track}
  \end{align}
  then $(\cv{x}, u_1, \ldots, u_m)$ is a trajectory of~\eqref{eq:lids} solving the PTTP i.e.~$\cv{x}(0) = x_0$ and $\lim_{t \to \infty} \cv{x}(t) \cdot \cv{x}_r(t)^{-1} = e$.
  \begin{proof} It is essentially based on the following simple fact -- a kind of Leibniz rule for curves on $G$ -- which the reader can easily verify: given $\cv{x}_1, \cv{x}_2: \R \rarr G$ two smooth curves we have\footnote{For $x \in G$ we denote by $L_x$ (resp.~$R_x$) the left (resp.~right) translation map~$y \in G \mapsto x \cdot y \in G$ (resp.~$y \in G \mapsto y \cdot x \in G$).}
    \begin{align*}
      (\cv{x}_1 \cdot \cv{x}_2)' &= \dd R_{\cv{x}_2} \cv{x}_1' + \dd L_{\cv{x}_1} \cv{x}_2'.
    \end{align*}

    Let $(\cv{x}, u_1, \ldots, u_m)$ be defined by~\eqref{eq:x_track}-\eqref{eq:u_track}. Then
    \begin{align*}
      \cv{x}' = (\cv{z} \cdot \cv{x}_r)' = \dd R_{\cv{x}_r} \cv{z}' + \dd L_{\cv{z}} \cv{x}_r' = \dd R_{\cv{x}_r} \sum_{k = 1}^m v_k \Ad(\cv{x}_r) X_k(\cv{z}) + \dd L_{\cv{z}} \sum_{k = 1}^m u_k^r X_k (\cv{x}_r)
    \end{align*}
    where the first sum can be rewritten as
    \begin{align*}
      \dd R_{\cv{x}_r} \sum_{k = 1}^m v_k \Ad(\cv{x}_r) X_k(\cv{z}) &= \dd R_{\cv{x}_r} \sum_{k = 1}^m v_k (R_{\cv{x}_r^{-1}})_* X_k(\cv{z})\\
      &= \dd R_{\cv{x}_r} \sum_{k = 1}^m v_k \dd R_{\cv{x}_r^{-1}} X_k(R_{\cv{x}_r} \cv{z}) \\
      &= \dd R_{\cv{x}_r} \dd R_{\cv{x}_r^{-1}} \sum_{k = 1}^m v_k X_k(\cv{x}_r \cdot \cv{z}) \\
      &= \sum_{k = 1}^m v_k X_k(\cv{x})
    \end{align*}
    while the second is
    \begin{align*}
      \dd L_{\cv{z}} \sum_{k = 1}^m u_k^r X_k (\cv{x}_r) = \sum_{k = 1}^m u_k^r \dd L_{\cv{z}} X_k (\cv{x}_r) = \sum_{k = 1}^m u_k^r  X_k (L_{\cv{z}} \cv{x}_r) = \sum_{k = 1}^m u_k^r  X_k (\cv{x})
    \end{align*}
    Summing it up and using~\eqref{eq:u_track} we conclude that $\cv{x}$ solves~\eqref{eq:lids}. Moreover
    \begin{align*}
      \cv{x}(0) = \cv{z}(0) \cdot \cv{x}_r(0) = x_0 \cdot x_\infty^{-1} \cdot x_\infty = x_0
    \end{align*}
    and
    \begin{align*}
      \lim_{t \to \infty} \cv{x}(t) \cdot \cv{x}_r(t)^{-1} = \lim_{t \to \infty} \cv{z}(t) = e.
    \end{align*}
  \end{proof}
\end{Prop}

Thanks to Proposition~\ref{prop:mod_sys}, in order to solve the PTTP our main concern shall be, from now on, to find a trajectory $(\cv{z}, v_1, \ldots, v_m)$ of system~\eqref{eq:mod_sys} satisfying $\cv{z}(0) = x_0 \cdot x_\infty^{-1}$ and $\lim_{t \to \infty} \cv{z}(t) = e$: the solution $(\cv{x}, u_1, \ldots, u_m)$ of the PTTP for~\eqref{eq:lids} can thus be recovered from our knowledge of $(\cv{z}, v_1, \ldots, v_m)$ and $(\cv{x}_r, u_1^r, \ldots, u_m^r)$.

We define a \emph{Lyapunov-like function} $V: G \rarr \R$ by
\begin{align}
  V(x) &\dfn \tr \Ad(x), \quad x \in G, \label{eq:Vdef}
\end{align}
and an \emph{auxiliary vector field} $W: \R \times G \rarr TG$ by
\begin{align*}
  W(t,w) &\dfn \sum_{k = 1}^m a_k(t,w) \Ad(\cv{x}_r(t)) X_k(w), \quad (t,w) \in \R \times G,
\end{align*}
where
\begin{align}
  a_k(t,w) &\dfn \dd V \left( \Ad(\cv{x}_r(t)) X_k(w) \right), \quad k \in \{1, \ldots, m\}. \label{eq:akdef} 
\end{align}
Notice that $W$ is a time-dependent vector field which is \emph{not} left-invariant. The main reason for introducing it is the following: if $\cv{w}: \R \rarr G$ is one of its integral curves and if we define
\begin{align}
  v_k(t) &\dfn a_k(t, \cv{w}(t)), \quad t \in \R, \ k \in \{1, \ldots, m\}, \label{eq:vk}
\end{align}
then $(\cv{w}, v_1, \ldots, v_m)$ is a trajectory of our modified system~\eqref{eq:mod_sys}. Moreover, let us denote by $\mathcal{C}_W$ the set of critical points of $W$, that is:
\begin{align*}
  \mathcal{C}_W &\dfn \{ w \in G \st W(t,w) = 0, \ \forall t \in \R \}.
\end{align*}
Recall that one such critical point $w \in \mathcal{C}_W$ is a \emph{local attractor} if there exists $U \sset G$ a neighborhood of $w$ such that given any initial condition $(t_0, w_0) \in \R \times U$ and $\cv{w}: \R \rarr G$ the unique integral curve of $W$ satisfying $\cv{w}(t_0) = w_0$ then $\lim_{t \to \infty} \cv{w}(t) = w$.

The next result tells us that if the identity element of $G$ is a local attractor of the auxiliary vector field $W$ then we can solve the PTTP \emph{locally} near the target state $x_\infty = \cv{x}_r(0)$, and also provides a recipe to obtain the tracking trajectory $(\cv{x}, u_1, \ldots, u_m)$.
\begin{Prop} \label{prop:elocalattrac} Suppose that $e \in \mathcal{C}_W$ and is a local attractor for $W$. Then there exists $U_\infty \sset G$ a neighborhood of $x_\infty$ enjoying the following property: for every $x_0 \in U_\infty$ there exists $(\cv{x}, u_1, \ldots, u_m)$ a trajectory of~\eqref{eq:lids} such that $\cv{x}(0) = x_0$ and $\lim_{t \to \infty} \cv{x}(t) \cdot \cv{x}_r(t)^{-1} = e$. The trajectory $(\cv{x}, u_1, \ldots, u_m)$ can be obtained as follows: for $\cv{w}: \R \rarr G$ the unique integral curve of $W$ satisfying $\cv{w}(0) = x_0 \cdot x_\infty^{-1}$, define
  \begin{align*}
    \cv{x}(t) &\dfn \cv{w}(t) \cdot \cv{x}_r(t), \\
    u_k(t) &\dfn a_k (t, \cv{w}(t)) + u_k^r(t), \quad k \in \{1, \ldots, m\}.
  \end{align*}
  \begin{proof} Let $U \sset G$ be an attractive neighborhood of $e$. Then $U_\infty \dfn U \cdot x_\infty$ is clearly a neighborhood of $x_\infty$, and if $x_0 \in U_\infty$ then $x_0 \cdot x_\infty^{-1} \in U$, hence $\lim_{t \to \infty} \cv{w}(t) = e$. If $v_1, \ldots, v_m$ are as in~\eqref{eq:vk} then $(\cv{w}, v_1, \ldots, v_m)$ is a trajectory of the modified system~\eqref{eq:mod_sys}, so the conclusion follows from Proposition~\ref{prop:mod_sys}.
  \end{proof}
\end{Prop}
\section{Some results on stability} \label{sec:stability}

In this section we shall depart a little from the original setting for the PTTP and establish some technical results on the stability of time-dependent vector fields that will be needed in the next sections. Since the group structure here plays no role, we shall take a step back and work in the more general framework of smooth manifolds.

\begin{Rem} As pointed out by H.~B.~Silveira in personal communication, our approach in this section (see especially Proposition~\ref{prop:figualzero}) holds some connections with the periodic version of LaSalle's Invariance Principle~\cite{lasalle66} (for its use in a similar context see~\cite{spr14}). The proofs we present here are, nevertheless, self-contained.
\end{Rem}

Let $M$ be a smooth manifold, which for simplicity we assume to be compact, and $X: \R \times M \rarr TM$ a time-dependent vector field. Recall that given $(t_0, x_0) \in \R \times M$ its \emph{$\omega$-limit set}, $\omega_X(t_0, x_0)$, is the set of all $x \in M$ enjoying the following property: there exists an increasing sequence $t_n \to \infty$ such that $\cv{x}(t_n) \to x$, where $\cv{x}: \R \rarr M$ is the unique integral curve of $X$ satisfying $\cv{x}(t_0) = x_0$. Of course the compactness of $M$ ensures that the $\omega$-limit sets of $X$ are never empty.

\begin{Def} A continuous function $V: M \rarr \R$ is said to be \emph{non-decreasing along $X$} if for every integral curve $\cv{x}: \R \rarr M$ of $X$ the function $V \circ \cv{x}$ is non-decreasing.
\end{Def}
For instance, if $V \in C^\infty(M)$ satisfies $\dd V (X(t,x)) \geq 0$ everywhere then clearly $V$ is non-decreasing along $X$.
\begin{Prop} Let $V: M \rarr \R$ be continuous and non-decreasing along $X$. Then $V$ is constant on $\omega_X(t_0, x_0)$ for any $(t_0, x_0) \in \R \times M$.
  \begin{proof} Let $\cv{x}: M \rarr \R$ be the unique integral curve of $X$ satisfying $\cv{x}(t_0) = x_0$. For $j = 1,2$ let $x_j \in \omega_X(t_0, x_0)$ and take increasing sequences $\{ t_n^j \}_{n \in \N}$, $t_n^j \to \infty$, such that $\cv{x}(t_n^j) \to x_j$ as $n \to \infty$. By continuity, $V(\cv{x}(t_n^j)) \to V(x_j)$ and since $V \circ \cv{x}$ is non-decreasing we must have
    \begin{align*}
      V(\cv{x}(t_n^j)) &\leq V(x_j), \quad \forall n \in \N, \ j = 1,2.
    \end{align*}
    We first extract a subsequence $\{ t_{n_k}^2 \}_{k \in \N}$ of $\{ t_n^2 \}_{n \in \N}$ with the property that $t_k^1 \leq t_{n_k}^2$ for every $k \in \N$: again, since $V$ is non-decreasing along $X$ one gets
    \begin{align*}
      V(\cv{x}(t_k^1)) \leq V(\cv{x}(t_{n_k}^2)) \leq V(x_2), \quad \forall k \in \N.
    \end{align*}
    By letting $k \to \infty$ in the left-hand side of the inequality above we conclude that $V(x_1) \leq V(x_2)$, and hence the equality holds.
  \end{proof}
\end{Prop}

\begin{Cor} \label{cor:ndecimpllim}If $V: M \rarr \R$ be continuous and non-decreasing along $X$ then
  \begin{align*}
    \lim_{t \to \infty} V(\cv{x}(t)) &= V(x), \quad \forall x \in \omega_X(t_0, x_0),
  \end{align*}
  where $\cv{x}: M \rarr \R$ be the unique integral curve of $X$ satisfying $\cv{x}(t_0) = x_0$.
  \begin{proof} It suffices to prove that any increasing sequence $\{ t_n \}_{n \in \N}$, $t_n \to \infty$, admits a subsequence $\{ t_{n_k} \}_{k \in \N}$ such that $V(\cv{x}(t_{n_k})) \to V(x)$ as $k \to \infty$. And indeed, by compactness of $M$ there exist $\{ t_{n_k} \}_{k \in \N}$ subsequence of $\{ t_n \}_{n \in \N}$ and $y \in M$ such that $\cv{x}(t_{n_k}) \to y$, and by continuity $V(\cv{x}(t_{n_k})) \to V(y)$, as $k \to \infty$. Since obviously $y \in \omega_X(t_0, x_0)$ we have $V(x) = V(y)$ by the previous proposition, and the conclusion follows.
  \end{proof}
\end{Cor}
The last two results in this section do not assume compactness of $M$. We do, however, endow it with a Riemannian metric: below we denote by $\| \cdot \|$ the induced norm on each tangent space.

\begin{Lem} \label{lem:unifomegalim} Let $\cv{x}:\R \rarr M$ be a smooth curve such that
  \begin{align*}
    \lim_{t \rarr \infty} \|\cv{x}'(t)\| &= 0
  \end{align*}
  and $\{t_n\}_{n \in \N}$ be an increasing sequence such that $t_n \to \infty$ and $\cv{x}(t_n) \to x \in M$ as $n \to \infty$. Then
  \begin{align*}
    \lim_{n \rarr \infty} \cv{x}(t_n + \epsilon) &= x, \quad \forall \epsilon \in \R.
  \end{align*}
  \begin{proof} We may assume w.l.o.g.~that $M$ is connected, and let $d:M \times M \rarr \R$ be the distance function on $M$ induced by the Riemannian metric\footnote{I.e.~the distance between two given points is the infimum of the lengths of all piecewise smooth curves connecting them.}: we then must prove that
    \begin{align*}
      \lim_{n \rarr \infty} d(\cv{x}(t_n + \epsilon), x) &= 0
    \end{align*}
    whatever $\epsilon \in \R$. If we denote by $I(a,b) \sset \R$ the closed interval with endpoints $a,b \in \R$ then by definition of $d$ we have
     \begin{align*}
      d(\cv{x}(t_n + \epsilon), \cv{x}(t_n)) \leq \left|\int_{I(t_n, t_n + \epsilon)} \|\cv{x}'(t)\| \dd t\right| \leq \left(\sup_{t \in I(t_n, t_n + \epsilon)} \|\cv{x}'(t)\| \right) |\epsilon|
     \end{align*}
     which, we claim, goes to zero as $n \to \infty$. Indeed, given $\delta > 0$ there exists $R > 0$ such that $ \|\cv{x}'(t)\| < \delta$ for every $t > R$. Moreover, since $t_n \to \infty$ there exists $n_0 \in \N$ such that
    \begin{align*}
      n \geq n_0 \Longrightarrow \max\{t_n, t_n + \epsilon\} > R \Longrightarrow \sup_{t \in I(t_n, t_n + \epsilon)} \|\cv{x}'(t)\| < \delta
    \end{align*}
    thus proving our claim. Now for every $n \in \N$
    \begin{align*}
      d(\cv{x}(t_n + \epsilon), x) \leq d(\cv{x}(t_n + \epsilon), \cv{x}(t_n)) + d(\cv{x}(t_n), x) \longrightarrow 0
    \end{align*}
    since both terms go to zero.
  \end{proof}
\end{Lem}

\begin{Prop} \label{prop:figualzero} Let $\cv{x}:\R \rarr M$ be a smooth curve, $\{t_n\}_{n \in \N}$ an increasing sequence and $x \in M$ as in Lemma~\ref{lem:unifomegalim}. Let also $f:\R \times M \rarr \R$ be continuous, $T$-periodic (for some $T > 0$) and such that $\lim_{t \to \infty} f(t,\cv{x}(t)) = 0$. Then
  \begin{align*}
    f(s,x) &= 0, \quad \forall s \in \R.
  \end{align*}
  \begin{proof} Let $s \in \R$. For each $n \in \N$ select $l_n \in \Z$ such that
    \begin{align*}
      s_n &\dfn t_n - l_nT \in [0, T)
    \end{align*}
    hence the sequence $\{s_n\}_{n \in \N}$ admits a convergent subsequence, say
    \begin{align*}
      \lim_{k \to \infty} s_{n_k} &= \theta \in [0,T].
    \end{align*}
    We define
    \begin{align*}
      s_{n_k}^* &\dfn s_{n_k} - \theta + s\\
      t_{n_k}^* &\dfn t_{n_k} - \theta + s = s_{n_k} + l_{n_k}T - \theta + s = s_{n_k}^* + l_{n_k}T.
    \end{align*}
    for each $k \in \N$, so clearly $s_{n_k}^* \to s$. Applying Lemma~\ref{lem:unifomegalim} with $\epsilon \dfn -\theta + s$ one gets
    \begin{align*}
      \cv{x}(t_{n_k}^*) = \cv{x}(t_{n_k} + \epsilon) \longrightarrow x.
    \end{align*}
    Since $f$ is continuous and $T$-periodic we have 
    \begin{align*}
      f(s,x) = \lim_{k \to \infty} f(s_{n_k}^*, \cv{x}(t_{n_k}^*)) = \lim_{k \to \infty} f(t_{n_k}^* - l_{n_k}T, \cv{x}(t_{n_k}^*)) = \lim_{k \rarr \infty} f(t_{n_k}^*, \cv{x}(t_{n_k}^*))
    \end{align*}
    which is zero thanks to our last hypothesis on $f$ and the fact that $t_{n_k}^* \to \infty$.
  \end{proof}
\end{Prop}

Now back to the setup of Section~\ref{sec:pmpp}, we use our results above to prove:
\begin{Thm} \label{thm:akvale0emOmega} Let $(t_0,w_0) \in \R \times G$. If $w \in \omega_W(t_0,w_0)$ then
  \begin{align*}
    a_k(t,w) &= 0, \quad \forall t \in \R, \ \forall k \in \{1, \ldots, m\}.
  \end{align*}
  In particular, every $\omega$-limit point is a critical point of $W$.
\end{Thm}
Its proof depends on some auxiliary results that we prove below. First of all, we must obtain a more convenient expression for the functions $a_k$ defined in~\eqref{eq:akdef}.
\begin{Lem} \label{lem:dV} For $(x,v) \in TG$ we have\footnote{Given $X \in \gr{g}$ the adjoint map $\ad(X): \gr{g} \rarr \gr{g}$ is defined by $Y \in \gr{g} \mapsto [X,Y] \in \gr{g}$. By means of the canonical isomorphism $\gr{g} \cong T_eG$ it makes perfect sense to write $\ad(v)$ for $v \in T_eG$ -- as we do often throughout the text -- which we regard as a linear map $T_eG \rarr T_e G$. In that sense, $\ad: T_eG \rarr \gl(T_eG)$ is precisely the differential of the adjoint map $\Ad: G \rarr \GL(T_eG)$ at $e \in G$~\cite[Proposition~1.91]{knapp_lgbi}.\label{footnote:adAd}}
  \begin{align}
    \dd V_x v &= \tr \left\{ \Ad(x) \cdot \ad ( \dd L_{x^{-1}} v)\right\}. \label{eq:dVexpl}
  \end{align}
  Also, for each $k \in \{1, \ldots, m\}$:
  \begin{align}
    a_k(t,w) &= \tr \left\{\Ad(w) \cdot \ad\left(\Ad\left(\cv{x}_r(t)\right)X_k\right) \right\} \label{eq:akexplic}
  \end{align}
  for $(t,w) \in \R \times G$.
  \begin{proof} We start by showing that
    \begin{align}
      \dd \Ad_x v &= \Ad(x) \cdot \ad( \dd L_{x^{-1}}v), \quad \forall (x,v) \in TG. \label{eq:derAdxqq}
    \end{align}
    Indeed, notice that $\dd L_{x^{-1}}v \in T_eG$, which we then identify with an element of $\gr{g}$, thus making sense of~\eqref{eq:derAdxqq}. We consider the map $F \dfn \Ad \circ L_x: G \rarr \GL(\gr{g})$: by the chain rule we have, on the one hand,
    \begin{align*}
      \dd \Ad_xv &= \dd F_e \dd L_{x^{-1}} v.
    \end{align*}
    On the other hand, we can write
    \begin{align*}
      F(y) = \Ad(x \cdot y) = \Ad(x) \cdot \Ad(y), \quad y \in G,
    \end{align*}
    i.e.~$F = \Lambda \circ \Ad$ where
    \begin{align*}
      \TR{\Lambda}{T}{\gl(\gr{g})}{\Ad(x) \cdot T}{\gl(\gr{g})}
    \end{align*}
    is a linear map: for $\xi \in T_eG$ we have again by the chain rule
    \begin{align*}
      \dd F_e \xi = \dd \Lambda_{\Ad(e)} \dd \Ad_e \xi = \Lambda (\ad(\xi)) = \Ad(x) \cdot \ad(\xi)
    \end{align*}
    which for $\xi \dfn \dd L_{x^{-1}}v$ proves~(\ref{eq:derAdxqq}) thanks to our previous conclusions.

    Now, recalling that the map $\tr: \gl(\gr{g}) \rarr \R$ is linear and by definition $V = \tr \circ \Ad$, identity~\eqref{eq:dVexpl} follows immediately from~(\ref{eq:derAdxqq}) after a third application of the chain rule.

    To conclude, it follows from the definition of $a_k$ and from~(\ref{eq:dVexpl}) that
    \begin{align*}
      a_k(t,w) &= \dd V(\Ad(\cv{x}_r)X_k(w))\\
      &= \tr \left\{\Ad(w) \cdot \ad\left(\dd L_{w^{-1}}\Ad(\cv{x}_r)X_k(w)\right)\right\}\\
      &= \tr \left\{\Ad(w) \cdot \ad\left(\Ad(\cv{x}_r)X_k(e)\right)\right\}\\
      &= \tr \left\{\Ad(w) \cdot \ad\left(\Ad(\cv{x}_r)X_k\right)\right\}
    \end{align*}
    where we used that $\Ad\left(\cv{x}_r(t)\right)X_k$ is left-invariant for all $t \in \R$.
  \end{proof}
\end{Lem}

We can now elucidate a couple of questions raised by Proposition~\ref{prop:elocalattrac}. 
\begin{Cor} \label{cor:einEW} Every\footnote{$Z(G)$: the center of $G$ i.e.~the subgroup of all $x \in G$ such that $x \cdot y = y \cdot x$ for every $y \in G$. When $G$ is semisimple $Z(G)$ is discrete.} $w \in Z(G)$ is a critical point of the auxiliary vector field $W$. However, if the identity element is a local attractor of $W$ then $G$ must be semisimple.
  \begin{proof} Since $Z(G) = \ker \Ad$ we have $\Ad(w) = \text{id}_{\gr{g}}$, and then for each $k \in \{1, \ldots, m\}$
    \begin{align*}
      a_k(t,w) = \tr \left\{\Ad(w) \cdot \ad\left(\Ad(\cv{x}_r)X_k\right)\right\} = \tr \ad\left(\Ad(\cv{x}_r)X_k\right) = 0
    \end{align*}
    for every $t \in \R$ since $\ad(X)$ is traceless\footnote{See footnote~\ref{footnote:aiip}.} for all $X \in \gr{g}$. By definition of $W$ we have then $W(t,w) = 0$ for all $t \in \R$ i.e.~$w$ is a critical point.

    In particular $e \in \mathcal{C}_W$. If $G$ were not semisimple then $Z(G)$ would be a Lie subgroup of $G$ of positive dimension, hence any neighborhood of $e$ would contain infinitely many points in $Z(G)$. Since $Z(G) \sset \mathcal{C}_W$ this proves that $e$ would not be an isolated point of $\mathcal{C}_W$, even less a local attractor.
  \end{proof}
\end{Cor}

The next technical remark will also be needed in Section~\ref{sec:existence}. We define
\begin{align}
  X_r &\dfn \sum_{j = 1}^m u_j^rX_j \label{eq:Xr}
\end{align}
where $u_1^r, \ldots, u_m^r$ are the controls of our reference trajectory of system~(\ref{eq:lids}). We will consider $X_r$ both as a time-dependent vector field on $G$ -- of which $\cv{x}_r$ is an integral curve -- and as a smooth curve in $\gr{g}$.
\begin{Lem} \label{lem:derlambda} Let $\Lambda^0:\R \rarr \gr{g}$ be any smooth curve and define $\lambda:\R \rarr \gr{g}$ by
  \begin{align*}
    \lambda &\dfn \Ad(\cv{x}_r)\Lambda^0.
  \end{align*}
  Then its $n$-th derivative is given by
  \begin{align*}
    \lambda^{(n)} &= \Ad(\cv{x}_r)\Lambda^n
  \end{align*}
  where $\Lambda^n:\R \rarr \gr{g}$ is defined inductively by
  \begin{align*}
    \Lambda^{n + 1} &\dfn (\Lambda^n)' + \ad(X_r)\Lambda^n, \quad n \in \N.
  \end{align*}
  \begin{proof} Using the identity $\Ad(\cv{x}_r)' = \Ad(\cv{x}_r) \cdot \ad(X_r)$, which in turn follows easily from~\eqref{eq:derAdxqq}, we have
    \begin{align*}
      \left(\Ad(\cv{x}_r)\Lambda^n\right)' = \Ad(\cv{x}_r)'\Lambda^n + \Ad(\cv{x}_r)(\Lambda^n)' = \Ad(\cv{x}_r)\ad(X_r)\Lambda^n + \Ad(\cv{x}_r)(\Lambda^n)' = \Ad(\cv{x}_r)\Lambda^{n + 1}.
    \end{align*}
  \end{proof}
\end{Lem}

\begin{Prop} \label{prop:bklim} Let $\cv{w}: \R \rarr G$ be an integral curve of $W$. For each $k \in \{1, \ldots, m\}$ the function $b_k:\R \rarr \R$ defined by
  \begin{align*}
    b_k(t) &\dfn \frac{\dd}{\dd t} a_k \left(t,\cv{w}(t)\right)
  \end{align*}
  is bounded.
  \begin{proof} We shall write down an explicit expression for $b_k$, which boils down to computing the partial derivatives of $a_k$ since
    \begin{align*}
      b_k &= \frac{\del a_k}{\del t} + \frac{\del a_k}{\del w} \cv{w}'.
    \end{align*}
    
    First, since $\Ad(x):\gr{g} \rarr \gr{g}$ is a Lie algebra homomorphism for each $x \in G$, it follows easily that
    \begin{align}
      \ad\left(\Ad(x)X\right) &= \Ad(x) \cdot \ad(X) \cdot \Ad(x)^{-1} \label{eq:adAdxX}
    \end{align}
    for every $X \in \gr{g}$ (just apply both sides on an arbitrary $Y \in \gr{g}$). Now, by~\eqref{eq:akexplic} we have that $a_k(t,w) = \tr\left\{\Ad(w) \cdot \ad\left(\lambda \right)\right\}$ where $\lambda \dfn \Ad(\cv{x}_r)X_k$, hence
    \begin{align*}
      \frac{\del a_k}{\del t} = \tr \left\{\Ad(w) \cdot \ad\left(\lambda'\right) \right\} = \tr \left\{\Ad(w) \cdot \Ad(\cv{x}_r) \cdot \ad\left( \ad(X_r) X_k\right) \cdot \Ad(\cv{x}_r)^{-1} \right\}
    \end{align*}
    by~\eqref{eq:adAdxX}, since $\lambda' = \Ad(\cv{x}_r)\ad(X_r)X_k$ thanks to Lemma~\ref{lem:derlambda}.
    
    Moreover, using~\eqref{eq:derAdxqq} and taking into account that $\cv{w}$ is an integral curve of $W$
    \begin{align*}
      \dd \Ad_{\cv{w}}\cv{w}' &= \Ad(\cv{w}) \cdot \ad\left(\dd L_{\cv{w}^{-1}}\cv{w}'\right)\\
      &= \Ad(\cv{w}) \cdot \ad\left(\dd L_{\cv{w}^{-1}}\sum_{j = 1}^m a_j(t, \cv{w})\Ad(\cv{x}_r)X_j(\cv{w})\right)\\
      &= \Ad(\cv{w}) \cdot \sum_{j = 1}^m a_j(t, \cv{w})\ \ad\left(\Ad(\cv{x}_r)X_j\right)\\
    \end{align*}
    from which it follows that
    \begin{align*}
      \frac{\del a_k}{\del w} \cv{w}' &= \tr \left\{\dd \Ad_{\cv{w}}\cv{w}' \cdot \ad\left(\Ad(\cv{x}_r)X_k\right)\right\}\\
      &= \tr \left\{ \Ad(\cv{w}) \cdot \sum_{j = 1}^m a_j(t, \cv{w})\ \ad\left(\Ad(\cv{x}_r)X_j\right) \cdot \ad\left(\Ad(\cv{x}_r)X_k\right)\right\}\\
      &= \tr \left\{ \Ad(\cv{w}) \cdot \Ad(\cv{x}_r) \cdot \sum_{j = 1}^m a_j(t, \cv{w}) \ \ad(X_j) \cdot \ad(X_k) \cdot \Ad(\cv{x}_r)^{-1}\right\}
    \end{align*}
    thanks again to a double application of~\eqref{eq:adAdxX}
    
    Summing both derivatives evaluated at $(t, \cv{w}(t))$, we conclude that
    \begin{align*}
      b_k &= \tr \left\{\Ad(\cv{w}) \cdot \Ad(\cv{x}_r) \cdot B_k \cdot \Ad(\cv{x}_r)^{-1}\right\}
    \end{align*}
    where $B_k:\R \rarr \gl(\gr{g})$ is defined given by
    \begin{align*}
      B_k &\dfn \ad(\ad(X_r)X_k) + \sum_{j = 1}^m a_j(t, \cv{w})\ad\left(X_j\right) \cdot \ad(X_k).
    \end{align*}
    Denoting by $\| \cdot \|$ any norm in $\gl(\gr{g})$, it follows from the compactness of $G$ the existence of $M > 0$ such that $\| \Ad(x)\| \leq M$ for every $x \in G$, hence for every $t \in \R$ we have
    \begin{align*}
      |b_k(t)| = \left|\tr \left\{\Ad(\cv{w}) \cdot \Ad(\cv{x}_r) \cdot B_k(t) \cdot \Ad(\cv{x}_r)^{-1} \right\}\right| \leq M^3 \|\tr\| \|B_k(t)\|
    \end{align*}
    where $\|\tr\|$ stands for the norm of the linear functional $\tr: \gl(\gr{g}) \rarr \R$: in order to finish the proof, it suffices to show that $B_k$ is bounded. But
    \begin{align*}
      \|B_k(t)\| &\leq \|\ad(\ad(X_r(t))X_k)\| + \sum_{j = 1}^m |a_j(t, \cv{w})| \ \|\ad\left(X_j\right) \cdot \ad(X_k)\|
    \end{align*}
    and while the first term is clearly bounded for the map
    \begin{align*}
      t \in \R &\longmapsto \ad(\ad(X_r(t))X_k) \in \gl(\gr{g}),
    \end{align*}
    is $T$-periodic, the second term is bounded because $a_j:\R \times G \rarr \R$ is $T$-periodic for each $j \in \{1, \ldots, n\}$ and hence $a_j(\R \times G) = a_j([0,T] \times G)$ is a compact set.
  \end{proof}
\end{Prop}

\begin{Cor} \label{cor:akconv0} If $\cv{w}:\R \rarr G$ is an integral curve of $W$ then
  \begin{align}
    \lim_{t \rarr \infty} a_k\left(t,\cv{w}(t)\right) &= 0, \quad \forall k \in \{1, \ldots, m\}. \label{eq:akconv0}
  \end{align}
  In particular
  \begin{align*}
    \lim_{t \to \infty} \|\cv{w}'(t)\| &= 0
  \end{align*}
  where $\| \cdot \|$ is the norm associated to any left-invariant Riemannian metric on $G$.
  \begin{proof}  Let $\alpha \dfn V \circ \cv{w}$ where $V$ is our Lyapunov-like function~\eqref{eq:Vdef}. Its first derivative is
    \begin{align*}
      \alpha' = \dd V(\cv{w}') = \dd V \left(\sum_{k = 1}^m a_k(t, \cv{w}) \Ad(\cv{x}_r)X_k(\cv{w})\right) = \sum_{k = 1}^m a_k(t, \cv{w}) \dd V(\Ad(\cv{x}_r)X_k(\cv{w})) = \sum_{k = 1}^m a_k(t, \cv{w})^2,
    \end{align*}
    by definition of $a_k$~\eqref{eq:akdef}, and thus non-negative. Differentiating once again yields
    \begin{align*}
      \alpha'' =  2\sum_{k = 1}^m a_k(t, \cv{w}) \frac{\dd}{\dd t} a_k(t, \cv{w}) = 2\sum_{k = 1}^m a_k(t, \cv{w}) b_k(t)
    \end{align*}
    with $b_k$ is as in Proposition~\ref{prop:bklim}, hence bounded, which implies boundedness of $\alpha''$. In turn, this ensures that $\alpha'$ is uniformly continuous. Now
    \begin{align*}
      \dd V(W(t,w)) &= \sum_{k = 1}^m a_k(t,w)^2 \geq 0, \quad \forall (t,w) \in \R \times G
    \end{align*}
    so $V$ is non-decreasing along $W$, thus thanks to Corollary~\ref{cor:ndecimpllim} we have
    \begin{align*}
      \lim_{t \to \infty} V(\cv{w}(t)) &= V(w)
    \end{align*}
    where $w \in \omega_W(0, \cv{w}(0))$ is arbitrary (recall that the latter set is never empty). We have proved that
    \begin{align*}
      \lim_{t \to \infty} \int_0^t \alpha'(s) \dd s = \lim_{t \to \infty} \alpha(t) - \alpha(0) = V(w) - V(\cv{w}(0))
    \end{align*}
    which brings us into position to apply Barbalat's Lemma~\cite[Lemma~8.2]{khalil} to $\alpha'$ and conclude that
    \begin{align*}
      \lim_{t \to \infty} \alpha'(t) &= 0
    \end{align*}
    which clearly proves~\eqref{eq:akconv0} thanks to our previous computations. Our second statement now follows:
     \begin{align*}
       \|\cv{w}'(t)\| = \left\|\sum_{k = 1}^m a_k(t,\cv{w}) \Ad(\cv{x}_r)X_k(\cv{w}) \right\| \leq \sum_{k = 1}^m |a_k(t,\cv{w})| \  \|\Ad(\cv{x}_r)X_k\| \leq M \sum_{k = 1}^m |a_k(t,\cv{w})| \  \|X_k\| \longrightarrow 0.
     \end{align*}
  \end{proof}
\end{Cor}

\begin{proof}[Proof of Theorem~\ref{thm:akvale0emOmega}] Let $\| \cdot \|$ stand for the norm associated to some left-invariant metric on $G$. For $\cv{w}:\R \rarr G$ the unique integral curve of $W$ satisfying $\cv{w}(t_0) = w_0$ we have, thanks to Corollary~\ref{cor:akconv0}, that $\|\cv{w}'(t)\| \to 0$ as $t \to \infty$. Moreover, the function $f:\R \times G \rarr \R$ defined by $f(t,w) \dfn a_k(t,w)$ is $T$-periodic and satisfies, again by Corollary~\ref{cor:akconv0},
  \begin{align*}
    \lim_{t \rarr \infty} f(t, \cv{w}(t)) = \lim_{t \rarr \infty} a_k(t, \cv{w}(t)) = 0.
  \end{align*}
  But since $w \in \Omega_W(t_0,w_0)$ there exists a sequence $\{t_n\}_{n \in \N}$ increasing to infinity such that $\cv{w}(t_n) \to w$ as $n \to \infty$. The conclusion follows from Proposition~\ref{prop:figualzero}.
\end{proof}
\section{Regular trajectories} \label{sec:regtraj}

Up to this point, all the results obtained are valid for arbitrary periodic reference trajectories of~\eqref{eq:lids}. In this section we introduce a special class of trajectories such that the set of critical points of their associated auxiliary vector fields $W$ admit a nice algebraic description: it coincides with the set of critical points of our Lyapunov-like function $V$. This characterization allows us show that the identity element is a local attractor for $W$ provided $G$ is semisimple, hence solving the PTTP in a neighborhood of the target state $x_\infty$ by Proposition~\ref{prop:elocalattrac}.

\begin{Def} \label{def:traj_regular} A trajectory $(\cv{x}, u_1, \ldots, u_m)$ of~\eqref{eq:lids} (not necessarily periodic) is said to be \emph{regular} if
  \begin{align*}
    \Span \left\{\Ad\left(\cv{x}(t)\right)X_k \st t \in \R, \ 1 \leq k \leq m\right\} &= \gr{g}.
  \end{align*}
\end{Def}
In Section~\ref{sec:existence} we prove the existence of regular periodic trajectories through any initial state $x_0 \in G$ under some extra assumptions on system~\eqref{eq:lids}. Theorem~\ref{thm:akvale0emOmega} admits the following:
\begin{Cor} \label{cor:characCW} Assume $(\cv{x}_r, u_1^r, \ldots, u_m^r)$ is a regular periodic trajectory of system~\eqref{eq:lids} and let $W$ be its auxiliary vector field. If $w \in G$ is any $\omega$-limit point of $W$ then
  \begin{align*}
    \tr \left\{\Ad(w) \cdot \ad(X)\right\} &= 0, \quad \forall X \in \gr{g}.
  \end{align*}
  Or, by Lemma~\ref{lem:dV}: $\dd V X(w) = 0$ for all $X \in \gr{g}$.
  \begin{proof} By Theorem~\ref{thm:akvale0emOmega} we have $a_k(t,w) = 0$ for every $t \in \R$ and $k \in \{1, \ldots, m\}$, so by Lemma~\ref{lem:dV}
    \begin{align*}
      \tr \left\{\Ad(w) \cdot \ad\left(\Ad\left(\cv{x}_r(t)\right)X_k\right)\right\} &= 0, \quad \forall t \in \R, \ \forall k \in \{1, \ldots, m\}.
    \end{align*}
    But since $(\cv{x}_r, u_1^r, \ldots, u_m^r)$ is regular the linear functional $\tr \left\{\Ad(w) \cdot \ad(\cdot)\right\}$ vanishes on $\gr{g}$.
  \end{proof}
\end{Cor}
In particular, for a regular periodic trajectory $(\cv{x}_r, u_1^r, \ldots, u_m^r)$ the set of critical points of its auxiliary field $W$ can be expressed as
\begin{align}
  \mathcal{C}_W &= \left\{ w \in G \st \tr \left\{\Ad(w) \cdot \ad(X)\right\} = 0, \ \forall X \in \gr{g} \right\}. \label{eq:EWcomocpV} 
\end{align}
Indeed, if $w$ belongs to the set in the right-hand side of~\eqref{eq:EWcomocpV} then by~\eqref{eq:akexplic}) we have $a_k(t,w) = 0$ for all $t \in \R$ and every $k \in \{ 1, \ldots, m\}$, hence clearly $w \in \mathcal{C}_W$. We have thus equated, in this case:
\begin{itemize}
\item the set of critical points of the vector field $W$,
\item the set of $\omega$-limit points of $W$ and
\item the set of critical points of the Lyapunov-like function $V$.
\end{itemize}

The next result gathers some interesting consequences of~\eqref{eq:EWcomocpV}, which, however, we will not use in what follows.
\begin{Prop} \label{prop:algebEW} Let $(\cv{x}_r, u_1^r, \ldots, u_m^r)$ be a regular periodic trajectory of~\eqref{eq:lids} and $W$ be its auxiliary vector field. Then for $x, y \in G$:
  \begin{enumerate}
  \item $x \in \mathcal{C}_W \Longrightarrow y \cdot x \cdot y^{-1} \in \mathcal{C}_W$.
  \item $x \cdot y \in \mathcal{C}_W \Longrightarrow y \cdot x \in \mathcal{C}_W$.
  \item $x \in \mathcal{C}_W \Longrightarrow x^{-1} \in \mathcal{C}_W$.
  \end{enumerate}
  Moreover, if $G$ is semisimple and $\mathcal{C}_W$ is a subgroup of $G$ then $\mathcal{C}_W$ is finite. 
  \begin{proof} \hfill
    \begin{enumerate}
    \item For $X \in \gr{g}$ a simple computation shows that
      \begin{align*}
        \tr \left\{\Ad\left(y \cdot x \cdot y^{-1}\right) \cdot \ad(X)\right\} = \tr \left\{\Ad(x) \cdot \ad \left( \Ad(y)^{-1}X \right)\right\} = 0
      \end{align*}
      since $x$ belongs to $\mathcal{C}_W$, hence so does $y \cdot x \cdot y^{-1}$.
    \item Follows from the previous item since $y \cdot (x \cdot y) \cdot y^{-1} = y \cdot x$.
    \item Pick $Y_1, \ldots, Y_n$ an orthonormal basis of $\gr{g}$ w.r.t.~some $\Ad$-invariant\footnote{An inner product on $\gr{g}$ is said to be \emph{$\Ad$-invariant} if $\Ad(x)$ is orthogonal w.r.t.~it for every $x \in G$. Such an inner product always exists when $G$ is compact, and thanks to the relationship between the adjoint maps (see footnote~\ref{footnote:adAd}) one also has that $\ad(X)$ is skew-symmetric w.r.t.~it for every $X \in \gr{g}$~\cite[Proposition~4.24]{knapp_lgbi}. In particular, $\tr \ad(X) = 0$ for every $X \in \gr{g}$.\label{footnote:aiip}} inner product $\langle \cdot, \cdot \rangle$. Then:
      \begin{align*}
        \tr \left\{\Ad\left(x^{-1}\right) \cdot \ad(X)\right\} &= \sum_{j = 1}^n \left \langle \Ad\left(x^{-1}\right) \ad(X)Y_j, Y_j \right \rangle \\
        &= - \sum_{j = 1}^n \left \langle Y_j, \ad(X)\Ad(x)Y_j \right \rangle \\
        &= - \tr \left\{\Ad(x) \cdot \ad(X)\right\}
      \end{align*}
      which equals $0$ if $x \in \mathcal{C}_W$. Since $X \in \gr{g}$ is arbitrary, $x^{-1} \in \mathcal{C}_W$.
    \end{enumerate}

    As for the second part of the statement, since $\mathcal{C}_W$ is a closed set it is a Lie subgroup of $G$: let $\gr{h} \sset \gr{g}$ be its Lie algebra. Let $X \in \gr{h}$ and for each $Y \in \gr{g}$ define $f_Y:\R \rarr \R$ by\footnote{Here and below we denote by $\e^X \in G$ the exponential of $X \in \gr{g}$.}
    \begin{align*}
      f_Y(t) &\dfn \tr \left\{ \Ad \left(\e^{tX}\right) \cdot \ad(Y)\right\}, \quad t \in \R.
    \end{align*}
    Since then $\e^{tX} \in \mathcal{C}_W$ we have $f_Y(t) = 0$ for every $t \in \R$, and thus
    \begin{align*}
      f_Y'(t) &= \frac{\dd}{\dd t} \tr \left\{ \Ad \left(\e^{tX}\right) \cdot \ad(Y)\right\}\\
      &= \tr \left\{ \frac{\dd}{\dd t}\Ad \left(\e^{tX}\right) \cdot \ad(Y)\right\}\\
      &= \tr \left\{\Ad\left(\e^{tX}\right) \cdot \ad\left(\dd L_{\e^{-tX}} \frac{\dd}{\dd t}\e^{tX}\right) \cdot \ad(Y)\right\}\\
      &= \tr \left\{\Ad\left(\e^{tX}\right) \cdot \ad\left(\dd L_{\e^{-tX}} X\left(\e^{tX}\right)\right) \cdot \ad(Y)\right\}\\
      &= \tr \left\{\Ad\left(\e^{tX}\right) \cdot \ad\left(X\right) \cdot \ad(Y)\right\}
    \end{align*}
    also equals $0$ for all $t \in \R$, in particular for $t = 0$: we have thus proved that $\tr \left\{\ad\left(X\right) \cdot \ad(Y)\right\} = 0$ for every $Y \in \gr{g}$. But this is the Killing form of $\gr{g}$, which is non-degenerate since we are assuming $G$ semisimple\footnote{This is Cartan's Criterion for Semisimplicity~\cite[Theorem~1.45]{knapp_lgbi}: the Killing form of $\gr{g}$ is the bilinear form $B: \gr{g} \times \gr{g} \rarr \gr{g}$ defined by $B(X,Y) \dfn \tr \{ \ad(X) \cdot \ad(Y) \}$. It is always negative semidefinite, while non-degenerate precisely when $G$ is semisimple.}, from which we conclude that $X = 0$. Since $X \in \gr{h}$ is arbitrary we have $\gr{h} = \{0\}$ i.e.~$\mathcal{C}_W$ is a discrete subgroup of $G$. Since $G$ is compact and $\mathcal{C}_W$ is closed the latter must be finite.
  \end{proof}
\end{Prop}
As we have seen, $e \in \mathcal{C}_W$ and if this set is finite then $e$ is an isolated point, which is a necessary condition for $e$ to be a local attractor. Next we shall focus on proving the latter property without relying on the assumption of $\mathcal{C}_W$ being a group.

\begin{Prop} \label{prop:xZGndeg} If $G$ is semisimple then every $x \in Z(G)$ is a non-degenerate critical point of $V$, and, in particular, an isolated point in $\mathcal{C}_W$.
  \begin{proof} By Corollary~\ref{cor:einEW}, every $x \in Z(G)$ belongs to $\mathcal{C}_W$, hence is a critical point of $V$ by our characterization of the latter set following Corollary~\ref{cor:characCW} (since $(\cv{x_r}, u_1^r, \ldots, u_m^r)$ is regular). As such, we check its non-degeneracy by computing the Hessian matrix of $V$ in convenient coordinates around $x$.

    We denote by $B$ the Killing form of $\gr{g}$. Since $G$ is assumed semisimple, $-B$ is an inner product on $\gr{g}$ and we denote by $Y_1, \ldots, Y_n \in \gr{g}$ an orthonormal basis w.r.t.~it to introduce the so-called coordinates of second kind: let $\varphi:\R^n \rarr G$ be defined by
    \begin{align*}
      \varphi(s_1, \ldots, s_n) &\dfn x \cdot \e^{s_1Y_1} \cdot \ldots \cdot \e^{s_nY_n}, \quad (s_1, \ldots, s_n) \in \R^n.
    \end{align*}
    Simple computations show that
    \begin{align*}
      \dd \varphi \left( \left. \frac{\del}{\del s_j} \right|_{s = 0} \right) &= Y_j(x), \quad \forall j \in \{1, \ldots, n\},
    \end{align*}
    hence $\varphi$ is a local diffeomorphism near $s = 0$. Moreover
    \begin{align*}
      V(\varphi(s_1, \ldots, s_n)) &= \tr \left\{\Ad\left(x \cdot \e^{s_1Y_1} \cdot \ldots \cdot \e^{s_nY_n}\right)\right\}\\
      &= \tr \left\{\Ad(x) \cdot \Ad \left(\e^{s_1Y_1}\right) \cdot \ldots \cdot \Ad\left(\e^{s_nY_n}\right)\right\}\\
      &= \tr \left\{\Ad(x) \cdot \e^{\ad(s_1Y_1)} \cdot \ldots \cdot \e^{\ad(s_nY_n)}\right\}\\
      &= \tr \left\{\e^{s_1\ad(Y_1)} \cdot \ldots \cdot \e^{s_n\ad(Y_n)}\right\}\\
    \end{align*}
    where we used two well-known facts: that $Z(G)$ is precisely the kernel of the $\Ad$ homomorphism; and the identity~\cite[Proposition~1.91]{knapp_lgbi}:
    \begin{align*}
      \Ad (\e^X) &= \e^{\ad(X)}, \quad \forall X \in \gr{g}.
    \end{align*}
    Another simple computation then shows that
    \begin{align*}
      \left. \frac{\del^2 (V \circ \varphi)}{\del s_j \del s_k} \right|_{s = 0} = \tr \left\{\ad(Y_j) \cdot \ad(Y_k)\right\} = B(Y_j,Y_k) = - \delta_{jk}, \quad \forall j,k \in \{1, \ldots, n\}
    \end{align*}
    thus showing that the Hessian matrix of $V$ at $x = \varphi(0)$ is non-degenerate. The last claim follows from Morse Lemma.
  \end{proof}
\end{Prop}

Next we characterize the center of $G$ in terms of $V$ regardless of semisimplicity. Let $\gr{g}_\C$ denote the complexification of $\gr{g}$ and let $\Ad_\C(x): \gr{g}_\C \rarr \gr{g}_\C$ be the complexification of $\Ad(x)$. Therefore
\begin{align*}
  \tr \Ad(x) = \tr \Ad_\C(x) = \sum_{j = 1}^n \lambda_j(x)
\end{align*}
where $\lambda_1(x), \ldots, \lambda_n(x) \in \C$ are the eigenvalues of $\Ad_\C(x)$, repeated according to their multiplicities. Let $\langle \cdot, \cdot \rangle$ be any $\Ad$-invariant inner product on $\gr{g}$ (recall that $G$ is compact) and let $\langle \cdot, \cdot \rangle_\C$ stand for its sesquilinear extension to $\gr{g}_\C$, which is then a Hermitian inner product on $\gr{g}_\C$. Clearly $\Ad_\C(x)$ is unitary w.r.t.~$\langle \cdot, \cdot \rangle_\C$, in particular it is a diagonalizable map and
\begin{align*}
  |\lambda_j(x)| &= 1, \quad \forall j \in \{1, \ldots, n\}.
\end{align*}
It follows then that for every $x \in G$ we have
\begin{align*}
  V(x) \leq |V(x)| = |\tr \Ad_\C(x)| = \left| \sum_{j = 1}^n \lambda_j(x) \right| \leq \sum_{j = 1}^n |\lambda_j(x)| = n.
\end{align*}
Therefore, if $x \in Z(G)$ then $\Ad(x) = \mathrm{id}_{\gr{g}}$ and hence $V(x) = n$, and thus is a global maximum of $V$.

We claim that the converse is also true i.e.~if $V(x) = n$ then $x \in Z(G)$. We must prove that $\lambda_j(x) = 1$ for every $j \in \{1, \ldots, n\}$. Indeed we have
\begin{align*}
  n = \sum_{j = 1}^n \lambda_j(x) = \sum_{j = 1}^n \Re \lambda_j(x) + i\sum_{j = 1}^n \Im \lambda_j(x)
\end{align*}
which implies that
\begin{align*}
  \sum_{j = 1}^n \Re \lambda_j(x) = n, &\quad \sum_{j = 1}^n \Im \lambda_j(x) = 0. 
\end{align*}
Moreover
\begin{align*}
  \Re \lambda_j(x) \leq |\Re \lambda_j(x)| \leq |\lambda_j(x)| = 1, \quad \forall j \in \{1, \ldots, n\},
\end{align*}
so if $\Re \lambda_k(x) < 1$ for some $k \in \{1, \ldots, n\}$ then
\begin{align*}
  \sum_{j = 1}^n \Re \lambda_j(x) = \Re \lambda_k(x) + \sum_{j \neq k} \Re \lambda_j(x) < n
\end{align*}
which would lead us to a contradiction, hence $\Re \lambda_k(x) = 1$ for every $k \in \{1, \ldots, n\}$, and since $|\lambda_k(x)| = 1$ we must also have $\Im \lambda_k(x) = 0$ for every $k \in \{1, \ldots, n\}$. In particular, we have proved:
\begin{Lem} \label{lem:caracZGporV} Any $x \in G$ belongs to $Z(G)$ if and only if $V(x) = \dim \gr{g}$.
\end{Lem}
Now we can prove one of the main results of the present work. Recall that $W$ denotes the auxiliary vector fields associated to a \emph{regular} periodic reference trajectory.
\begin{Thm} \label{thm:final} If $G$ is semisimple then every $x \in Z(G)$ is a local attractor of $W$.
  \begin{proof} By Proposition~\ref{prop:xZGndeg} we may find a neighborhood $U \sset G$ of $x$ such that $\overline{U}$ is compact and contains no points in $\mathcal{C}_W$ other than $x$. We may further assume $U$ connected, and define
    \begin{align*}
      M &\dfn \max_{y \in \del U} V(y).
    \end{align*}
    Clearly $M < n$ for $\del U \cap Z(G) = \eset$. Also, since $V(x) = n$ we have by continuity that there exists $U' \sset U$ another neighborhood of $x$ such that $V(y) > M$ for every $y \in U'$.
    
    Now let $(t_0,w_0) \in \R \times U'$ and $\cv{w}:\R \rarr G$ be the unique integral curve of $W$ satisfying $\cv{w}(t_0) = w_0$. Since $V \circ \cv{w}$ is non-decreasing (for $V$ is non-decreasing along $W$, as pointed out earlier) we have
    \begin{align*}
      V(\cv{w}(t)) \geq V(\cv{w}(t_0)) = V(w_0) > M
    \end{align*}
    for every $t \geq t_0$ since $w_0 \in U'$ by hypothesis. Thus $\cv{w}(t) \notin \del U$ and by continuity we have that $\cv{w}(t) \in U$ for every $t \geq t_0$.

    It remains to show that $\cv{w}(t) \to x$ as $t \to \infty$. Indeed, we will show that any sequence $\{t_n\}_{n \in \N}$ increasing to infinity admits a subsequence $\{t_{n_k}\}_{k \in \N}$ such that $\cv{w}(t_{n_k}) \to x$ as $k \to \infty$. We may assume w.l.o.g.~that $t_n \geq t_0$ for every $n \in \N$, hence $\cv{w}(t_n) \in U$. Because $\overline{U}$ is compact there exists $w \in \overline{U}$ and a subsequence $\{t_{n_k}\}_{k \in \N}$ of $\{t_n\}_{n \in \N}$ such that $\cv{w}(t_{n_k}) \to w$ as $k \to \infty$. In particular $w \in \omega_W(t_0, w_0) \sset \mathcal{C}_W$, but as we have seen $\overline{U} \cap \mathcal{C}_W = \{ x \}$.
  \end{proof}
\end{Thm}

In particular $e$ is a local attractor of $W$, hence Proposition~\ref{prop:elocalattrac} solves the PTTP locally.

\section{Existence of regular periodic trajectories: sufficient conditions} \label{sec:existence}

Recall that we are always assuming that our left-invariant system~\eqref{eq:lids} is bracket-generating i.e.~the Lie algebra spanned by $X_1, \ldots, X_m$ is $\gr{g}$. In this section, we will prove that if moreover
\begin{align}
  \Span \{ \ad(X_1)^n X_k \st 1 \leq k \leq m, \ n \in \N \} &= \gr{g} \label{eq:reg_sys}
\end{align}
then given any $x_\infty \in G$ and $T > 0$ there exists $(\cv{x}_r, u_1^r, \ldots, u_m^r)$ a smooth $T$-periodic trajectory of~\eqref{eq:lids} satisfying $\cv{x}_r(0) = x_\infty$ and which is regular in the sense of Definition~\ref{def:traj_regular}.

For $j \in \{1, \ldots, m\}$ and $p \in \Z$ let
\begin{align*}
  J_j^p \dfn \left( \frac{(j - 1)T}{m} + pT, \frac{jT}{m} + pT\right) = J_j^0 + pT.
\end{align*}
Clearly $J_j^p \sset (pT, (p+1)T) = (0,T) + pT$ and
\begin{align}
  J_j^p \cap J_k^q \neq \eset &\Longleftrightarrow j = k, \ p = q. \label{eq:Jjdisj}
\end{align}
Take $\chi_j \in C_c^\infty(J_j^0)$ equal to $1$ in some open interval $I_j \sset J_j^0$ and with zero integral, and let $u_j^r:\R \rarr \R$ be the unique $T$-periodic function which is equal to $\chi_j$ on $[0,T]$: this is clearly smooth.

Since $\supp \chi_j \sset J_j^0$ it is easy to check that
\begin{align}
  \supp u_j^r \sset \bigcup_{p \in \Z} J_j^p \label{eq:suppujr}
\end{align}
which, together with~\eqref{eq:Jjdisj}, easily ensures the following:
\begin{Prop} \label{prop:sohumnaozero} Given $t \in \R$, if $u_j^r(t) \neq 0$ for some $j \in \{1, \ldots, m\}$ then $u_k^r(t) = 0$ for every $k \neq j$.
\end{Prop}
Thus on $I_j^p \dfn I_j + pT$ we have $u_j^r = 1$, while $u_k^r = 0$ for $k \neq j$, identically. Next, define $\xi_j:\R \rarr \R$ by 
\begin{align*}
  \xi_j(t) &\dfn \int_0^t u_j^r(s) \dd s
\end{align*}
which is obviously smooth and also $T$-periodic since $u_j^r$ is $T$-periodic and its integral over $[0,T]$ is zero. Moreover, one may check that
\begin{align}
  \supp \xi_j \sset \bigcup_{p \in \Z} J_j^p. \label{eq:suppxij}
\end{align}

We finally define $\cv{x}_r:\R \rarr G$ by
\begin{align*}
  \cv{x}_r(t) &\dfn 
  \begin{cases}
    \e^{\xi_j(t)X_j}(x_\infty), & \text{if $t \in \bigcup_{p \in \Z} J_j^p$, for $j = 1, \ldots, m$}; \\
    x_\infty, & \text{otherwise}.
  \end{cases}
\end{align*}
This is well-defined thanks to~\eqref{eq:Jjdisj}, and moreover smooth by~\eqref{eq:suppxij}. Also, on $I_j^p$ we have
\begin{align*}
  \cv{x}_r'(t) = \frac{\dd}{\dd t}\e^{\xi_j(t)X_j}(x_\infty) = \xi_j'(t) X_j(\e^{\xi_j(t)X_j}(x_\infty)) = u_j^r(t)X_j(\cv{x}_r(t)) = \sum_{k = 1}^m u_k^r(t)X_k(\cv{x}_r(t))
\end{align*}
where the last identity follows from Proposition~\ref{prop:sohumnaozero}. On the other hand, if $t \notin J_j^p$ for any $j \in \{ 1, \ldots, m\}$ and $p \in \Z$ then near $t$ we have $\cv{x}_r = x_\infty$ identically, hence $\cv{x}_r'(t) = 0$, which also agrees with~\eqref{eq:lids} thanks to~\eqref{eq:suppujr}: we have proved that $(\cv{x}_r, u_1^r, \ldots, u_m^r)$ is a trajectory of~\eqref{eq:lids}, which is $T$-periodic by construction.

\begin{Thm} If~\eqref{eq:reg_sys} holds then the trajectory $(\cv{x}_r, u_1^r, \ldots, u_m^r)$ above is regular.
  \begin{proof} We denote by
    \begin{align*}
      \mathcal{V} \dfn \Span \left\{\Ad\left(\cv{x}_r(t)\right)X_k \st t \in \R, \ 1 \leq k \leq m\right\}
    \end{align*}
    which we must prove that is equal to $\gr{g}$. For each $k \in \{1, \ldots, m\}$ let $\Lambda^0_k \dfn X_k$ and $\lambda_k:\R \rarr \gr{g}$ be defined by $\lambda_k \dfn \Ad(\cv{x}_r)\Lambda^0_k$: this is a smooth curve that lies in $\mathcal{V}$, and since the latter is a linear subspace of $\gr{g}$ the same is true for all of its derivatives. By Lemma~\ref{lem:derlambda} we have, for every $n \in \N$,
    \begin{align*}
      \lambda_k^{(n)} &= \Ad(\cv{x}_r)\Lambda_k^n\\
      \Lambda_k^{n + 1} &\dfn \left(\Lambda_k^n\right)' + \ad(X_r)\Lambda_k^n
    \end{align*}
    where $X_r$ is given by~\eqref{eq:Xr}.

    We need the following technical lemma, which does not depend on the construction of $(\cv{x}_r, u_1^r, \ldots, u_m^r)$.
    \begin{Lem} \label{lem:tech_final} For each $n \in \N$ we may write
      \begin{align*}
        \Lambda_k^n &= \Delta_nX_k + \ad(X_r)^nX_k
      \end{align*}
      where $\Delta_0 = \Delta_1 = 0$ and, for $n \geq 2$, $\Delta_n$ is a sum of products enjoying the following property: in each summand there is at least one factor that is a derivative of some order of $\ad(X_r)$. 
      \begin{proof}[Proof of Lemma~\ref{lem:tech_final}] By induction on $n$. We start calculating recursively by Lemma~\ref{lem:derlambda}
        \begin{align*}
          \Lambda_k^0 &= X_k\\
          \Lambda_k^1 &= \ad(X_r)X_k\\
          \Lambda_k^2 &= \ad(X_r)'X_k + \ad(X_r)^2X_k
        \end{align*}
        from which we identify $\Delta_0 = 0$, $\Delta_1 = 0$ and $\Delta_2 = \ad(X_r')$, thus proving our claim for $n = 0,1,2$, which we use as basis for induction. Assuming our conclusion for some $n \geq 2$ we have
        \begin{align*}
          \Lambda_k^{n + 1} &= \left(\Lambda_k^n\right)' + \ad(X_r)\Lambda_k^n\\
          &= \left(\Delta_nX_k + \ad(X_r)^nX_k\right)' + \ad(X_r)\left(\Delta_nX_k + \ad(X_r)^nX_k\right)\\
          &= \left(\Delta_n\right)'X_k + \left(\ad(X_r)^n\right)'X_k + \ad(X_r)\Delta_nX_k + \ad(X_r)^{n + 1}X_k\\
          &= \Delta_{n + 1}X_k + \ad(X_r)^{n + 1}X_k
        \end{align*}
        where obviously
        \begin{align*}
          \Delta_{n + 1} &\dfn \left(\Delta_n\right)' + \left(\ad(X_r)^n\right)' + \ad(X_r)\Delta_n.
        \end{align*}
        Since $\Delta_n$ is a sum of products in which each summand there is at least one factor that is a derivative of some order of $\ad(X_r)$ then of course the same property holds true for both $\left(\Delta_n\right)'$ and $\ad(X_r)\Delta_n$. Moreover
        \begin{align*}
          \left(\ad(X_r)^n\right)' &= \sum_{p = 1}^n \ad(X_r)^{p - 1} \cdot \ad(X_r)' \cdot \ad(X_r)^{n - p}
        \end{align*}
        also enjoys the aforementioned property, hence so does $\Delta_{n + 1}$.
      \end{proof}
    \end{Lem}
    
    Now, if $t_1 \in I_1^0$ then $u_1^r(t_1) = 1$, while for $j \neq 1$ we have by Proposition~\ref{prop:sohumnaozero} that $u_j^r$ vanishes identically near $t_1$, and therefore
    \begin{align*}
      X_r(t_1) = \sum_{j = 1}^m u_j^r(t_1) X_j = X_1
    \end{align*}
    while 
    \begin{align*}
      X_r^{(n)}(t_1) &= 0, \quad \forall n \geq 1.
    \end{align*}
    It follows from Lemma~\ref{lem:tech_final} that $\Delta_n(t_1) = 0$ (for $\ad:\gr{g} \rarr \gr{g}$ is linear, hence $\ad(X_r)^{(k)} = \ad \left(X_r^{(k)} \right)$ for every $k \in \N$) for every $n \in \N$, so
    \begin{align*}
      \Lambda_k^n(t_1) = \Delta_n(t_1)X_k + \ad(X_r(t_1))^nX_k = \ad(X_1)^nX_k
    \end{align*}
    for every $n \in \N$. We conclude that
    \begin{align*}
      \lambda_k^{(n)}(t_1) = \Ad(\cv{x}_r(t_1)) \Lambda_k^n(t_1) = \Ad(\cv{x}_r(t_1)) \ad(X_1)^nX_k
    \end{align*}
    belongs to $\mathcal{V}$ for every $n \in \N$ and $k \in \{1, \ldots, m\}$; in other words, if we denote
    \begin{align*}
      \mathcal{W} &\dfn \Span \{\ad(X_1)^nX_k \st 1 \leq k \leq m, \ n \in \N\}
    \end{align*}
    then $\Ad(\cv{x}_r(t_1)) \mathcal{W} \sset \mathcal{V}$. But we assumed in~\eqref{eq:reg_sys} that $\mathcal{W} = \gr{g}$ and $\Ad(\cv{x}_r(t_1))$ is invertible, hence also $\mathcal{V} = \gr{g}$ i.e.~$(\cv{x}_r, u_1^r, \ldots, u_m^r)$ is regular.
  \end{proof}
\end{Thm}

\def\cprime{$'$}

\end{document}